\documentclass[b5paper,reqno]{amsart}
\usepackage{amsfonts,lingmacros,tree-dvips, amsmath, amssymb, graphicx, mathrsfs}
\usepackage[hidelinks]{hyperref}
\usepackage{enumerate}
\usepackage[all]{xy}

\newdir{ >}{!/6pt/@{ }*:(1,0.2)@_{>}*:(1,-0.2)@^{>}}

\theoremstyle{plain}

\newtheorem{theorem}{Theorem}[section]

\newtheorem{proposition}[theorem]{Proposition}
\theoremstyle{definition}
\newtheorem{definition}[theorem]{Definition}

\newtheorem{counter example}[theorem]{Counter Example}

\numberwithin{equation}{section}

\DeclareMathAlphabet{\mathscr}{OT1}{pzc}{m}{it} 
\begin{document}
\Large{
		\title{ SUM OF TWO RARE SETS IN A CATEGORY BASE CAN BE ABSOLUTELY NON-bAIRE}
		
		\author[S. Basu]{Sanjib Basu}
		\address{\large{Department of Mathematics,Bethune College,181 Bidhan Sarani}}
		\email{\large{sanjibbasu08@gmail.com}}
		
		\author[A.C.Pramanik]{Abhit Chandra Pramanik}
	\address{\large{Department of Pure Mathematics, University of Calcutta, 35, Ballygunge Circular Road, Kolkata 700019, West Bengal, India}}
	\email{\large{abhit.pramanik@gmail.com}}

		\thanks{The second author thanks the CSIR, New Delhi – 110001, India, for financial support}
	\begin{abstract}
   In this paper, we give generalized version (in category bases) of a result of Kharazishvili dealing with absolute nonmeasurability of the Minkowski's sum of certain universal measure zero sets which were based on an earlier result of Erd\H{o}s, Kunen and Mauldin in the real line.
	\end{abstract}
\thanks{}
	\maketitle

\section{INTRODUCTION}
Sierpi\'nski [3] showed that there exists two Lebesgue Measure zero sets $A$ and $B$ such that their Minkowski (or, algebraic) sum $A+B$ is nonmeasurable. In proving this result, he used Hamel basis and also the fact that by using an interesting property of the Cantor set in the real line $\mathbb{R}$, it is possible to construct a measure zero, first category set $A$ such that $A+A=\mathbb{R}$. Kharazishvili later gave a generalization of this result [4] in the following manner:
\begin{theorem}
	Let $\mu$ be a non-zero, complete, $\sigma$-finite measure on $\mathbb{R}$ quasi-invariant under a group H of affine transformations. Suppose there are two measure zero sets $A$ and $B$ such that $A+B=\mathbb{R}$. Then there exists a $\mu$-measure zero set $X$ such that $X+X$ is not $\mu$-measurable.   
\end{theorem}
In a topological spaces $T$ in which all singletons are Borel sets, a subset $X$ is called a set of universal measure zero [5] if for every $\sigma$-finite, diffused Borel measure $\mu$ on $T$, $\mu^*(X)=0$, where $\mu^*$ is the outer measure induced by $\mu$.\\
An universal measure zero subset of $\mathbb{R}$ is always a set of Lebesgue measure zero but not conversely. For although the Cantor set $C$ is a set of Lebesgue measure zero, there exists a Borel probability measure $\nu$ such that $\nu(C)>0$ whose completion is isomorphic to the restriction of the Lebesgue measure $\lambda$ on[0,1].\\
Several constructions of universal measure zero sets in uncountable Polish topological spaces can be found in the literature [2], [7], [8] etc. But their existance cannot be established within ZFC. They are obtained either by assuming continuum hypothesis or Martin Axiom. The Luzin set [5] (see alos [3]) (whose construction depends on the continuum hypothesis) and the generalized Luzin set [5] (see also [3]) (whose construction depends on the Martin Axiom) in $\mathbb{R}$ are examples of universal measure zero sets. Here is the general definition in any uncountable Polish topological space.\\
A subset $X$ of a Polish space $T$ is called a Luzin set (or, a generalized Luzin set) if card($X$)=$c$ and for every first category set $Z$ in $T$, card($X\cap Z$)$<\aleph_0$ (or, card($X\cap Z$)$<c$) where $\aleph_0$, $c$ denote respectively the first infinite cardinal and the cardinality of the continuum.\\
The following theorem by Kharazishvili [6] (see also [5]) which is a modified version of an earlier result of Erd\H{o}s, Kunen and Mauldin [1] shows that generalized Luzin set in a Polish topological vector space can have additional algebraic properties :
\begin{theorem}
	Suppose that Martin Axiom holds. Then in any Polish topological vector space $E(\neq\phi)$, there exist sets $L_1$ and $L_2$ such that :
	\begin{enumerate}[(i)]
		\item both $L_1$ and $L_2$ are are generalized Luzin sets in $E$.
		\item both $L_1$ and $L_2$ are vector spaces over the field $\mathbb{Q}$ of rational numbers.
		\item $L_1+L_2=E$ and $L_1\cap L_2=\{0\}$.
	\end{enumerate}
\end{theorem} 
A subset $X$ of a topological space $T$ in which all singleton sets are Borel is said to be absolutely non-measurable [5] in $T$ if for every $\sigma$-finite, diffused Borel measure $\mu$ on $T$, we have $X\notin$dom($\mu'$) where $\mu'$ is the completion of $\mu$.\\
In uncountable polish spaces, there is a purely topological characterization of absolutely non-measurable sets. For this, we need to introduce here the notion of a Bernstein set [5] (see also [3]).\\
A subset $B$ of a topological space $T$ is called a Bernstein set if for every nonempty perfect set $P$ of $T$, $B\cap P\neq\phi\neq(T-B)\cap P$.\\
The construction of a Bernstein set [3] is based on transfinite induction. Such sets have many important applications in general topology, measure theory and real analysis.\\
The following topological characterization [5] of an absolutely non-measurable set is well known.
\begin{proposition}
	In an uncountable Polish space $T$, a subset $Z$ is absolutely non-measurable iff it is a Bernstein set.
\end{proposition} 
\newpage
Based on Theorem 1.2, Khrazishvili further proved [6] (see also [5]) that
\begin{theorem}
	Under Martin's Axiom, in any uncountable Polish topological vector spaces. there exist generalized Luzin sets $L_1$ and $L_2$ such that $L_1+L_2$ is a Bernstein set.
\end{theorem}
Luzin set [8] (see also [3]) in the real line was constructed by Luzin in 1914 using continuum hypothesis. Although a smiliar construction was given earlier in 1913 by Mahalo [8], in literature this set is commonly known as the Luzin set probabaly because Luzin investigated these sets more thoroughly and proved a number of its important properties. The construction of a generalized Luzin set (under Martin's Axiom) in $\mathbb{R}$ imitates the classical construction of Luzin.\\
The dual of a Luzin set in $\mathbb{R}$ is the Sierpi\'nski set [8] (see also [3]) constructed by Sierpi\'nski in 1924 using the same continuum hypothesis. It is an uncountable set of reals having countable intersection with every set of Lebesgue measure zero. The construction of a generalized Sierpi\'nski set (under Martin's construction) [8] (see also [3]) imitates the classical construction of Sierpi\'nski.\\
The concept of a Luzin set and its dual the Sierpi\'nski set can be unified under the general concept of `rare set' [9] in category bases. In this article, we formulate under continuum hypothesis, generalized versions of Theorem 1.2 and Theorem 1.4 in category bases with the underlying set having Polish topological vector space structure. Since under continuum hypothesis, a generalized rare set is a rare set, we state our results interms of rare sets instead of generalized rare sets.
\section{PRELIMINARIES AND RESULTS}
The concept of a category base is a generalization of both measure and topology. Its main objective is to present both measure and Baire category (topology) and also some other aspects of point set classification within a common framework. It was introduced by J.C.Morgan II [9] in the seventies of the last century and since then has been developed through a series of papers [10], [11], [12], [13] etc.\\
\begin{definition}[\cite{QM1979}]
A pair $(X,\mathcal{C}$) where $X$ is a non-empty set and $\mathcal{C}$ is a family of subsets of $X$ is called a category base if the non-empty members of $\mathcal{C}$ called regions satisfy the following set of axioms :
\begin{enumerate}
\item Every point of $X$ is contained in at least one region; i.e., $X=\cup$ $\mathcal{C}$.
\item Let $A$ be a region and $\mathcal{D}$ be any non-empty family of disjont regions having cardinality less than the cardinality of $\mathcal{C}$.
\begin{enumerate}[(i)]
	\item If $A \cap ( \cup \mathcal{D}$) contains a region, then there exists a region $D\in\mathcal{D}$ such that $A\cap D$ contains a region.
	\item If $A\cap(\cup \mathcal{D})$ contains no region, then there exists a region $B\subseteq A$ which is disjoint from every region in $\mathcal{D}$.
\end{enumerate}
\end{enumerate}
\end{definition}
Several examples of category bases may be found in [9] and also in other references as stated in the above paragraph.
\begin{definition}[\cite{QM1979}] 
	In a category base ($X,\mathcal{C}$), a set is called `singular' if every region contains a subregion which is disjoint from the set itself. A set which can be expressed as a countable union of singluar sets is called `meager'. Otherwise, it is called `abundant'. A set is called `Baire set' if every region contains a subregion in which either the set or its complement is meager. The class of all meager (resp, Baire) sets in a category base ($X,\mathcal{C}$) is here denoted by the symbol $\mathcal{M}(\mathcal{C})$ (resp, $\mathcal{B}(\mathcal{C})$).
\end{definition}
\begin{definition}[\cite{QM1979}]
	A category base ($X,\mathcal{C}$) is called 
	\begin{enumerate}[(i)]
		\item `point-meager' if every singleton set in it is meager.
		\item  a `Baire base' if every region in it is abundant.
		\item a `perfect base' if $X=\mathbb{R}^n$ and $\mathcal{C}$ is a family of perfect subsets of $\mathbb{R}^n$ such that for every region $A$ in $\mathcal{C}$ and for every point $x\in A$, there is a descending sequence \{$A_n\}_{n=1}^\infty$ of regions such that $x\in A_n$ and diam($A_n)\leq\frac{1}{n}$, $n=1,2,...$.\\	and	
		\item a `translation base' if $X=\mathbb{R}$ with 
		\begin{enumerate}[(a)]
			\item $\mathcal{C}$ as a translation invariant family, i.e., $E\in \mathcal{C}$ and $r\in\mathbb{R}$ implies $E(r)\in\mathcal{C}$ where $E(r)=\{x+r:x\in E\}$, and 
			\item if $A$ is any region and $D$ is a countable everywhere dense set in $\mathbb{R}$, then $\bigcup\limits_{r\in D}A(r)$ is abundant everywhere.
		\end{enumerate}
	\end{enumerate}
\end{definition}
To this, we also add that 
\begin{definition} [\cite{QM1979}]
	A category base ($X,\mathcal{C}$) 
	\begin{enumerate}[(i)]
		\item satisfies countable chain condition (c.c.c) if every subfamily of mutually disjoint regions is atmost countable.
		\item is linearly invariant if $X=\mathbb{R}$, $\mathcal{C}$ is translation invariant and for every $E\in \mathcal{C}$, $\lambda(\neq 0)\in \mathbb{R}$, we have $\lambda E\in \mathcal{C}$.  
	\end{enumerate} 
\end{definition}
\newpage
\begin{theorem}
	Suppose continuum hypothesis hold. Then in any point-meager base ($\mathbb{R},\mathcal{C}$)  satisfying c.c.c where $\mathcal{C}$ is a linearly invariant family having cardinality $\leq 2^{\aleph_0}$ and such that $\mathbb{R}$ is abundant, we can construct sets $E$ and $F$ such that 
	\begin{enumerate}[(i)]
		\item both $E$ and $F$ are rare sets in this category base.
		\item both $E$ and $F$ are vector spaces over the field $\mathbb{Q}$ of rationals.
		\item $\mathbb{R}=E+F$ and $E\cap F=\{0\}$.
	\end{enumerate}
\end{theorem}
\begin{proof}
	Let $\Omega$ denotes the least ordinal number representing $2^{\aleph_0}$ which is the cardinality of the continuum. Since $\mathcal{C}$ satisfies countable chain condition, every meager set in this category base is contained in a meager $\mathcal{K}_{\delta\sigma}$ set where $\mathcal{K}$ denotes the family of all sets whose complements are members of $\mathcal{C}$ (Th 5, II, Ch 1, [9]). Moreover, as the category base is point meager and given that the cardinality of $\mathcal{C}$ is atmost $2^{\aleph_0}$, by (Th 1, I, Ch 2, [9]), the cardinality of the family of all meager $\mathcal{K}_{\delta\sigma}$ sets is $2^{\aleph_0}$.\\
	Let \{$z_\alpha\}_{\alpha<\Omega}$ and \{$K_\alpha\}_{\alpha<\Omega}$ be enumeration of points and meager $\mathcal{K}_{\delta\sigma}$ sets in $\mathbb{R}$. Using transfinite recurssion, we construct two $\Omega$-sequences \{$X_\alpha\}_{\alpha<\Omega}$, \{$Y_\alpha\}_{\alpha<\Omega}$ of subsets of $\mathbb{R}$ in the following manner : \\
	Suppose, for any ordinal $\alpha<\Omega$, the partial $\alpha$-sequences \{$X_\beta:\beta<\alpha$\} and \{$Y_\beta:\beta<\alpha$\} are already constructed. Since ($\mathbb{R},\mathcal{C}$) is point-meager and $\mathcal{C}$ is linearly invariant, so under continuum hypothesis both $X'+\mathbb{Q}(\bigcup\limits_{\beta<\alpha}K_\beta)$ and $Y'+\mathbb{Q}(\bigcup\limits_{\beta<\alpha}K_\beta)$ are meager in $\mathbb{R}$ where $X'=\bigcup\limits_{\beta<\alpha}X_\beta$  and $Y'=\bigcup\limits_{\beta<\alpha}Y_\beta$. Hence, ($\mathbb{R}$\textbackslash $\{X'+\mathbb{Q}(\bigcup\limits_{\beta<\alpha}K_\beta)\})\cap(z_\alpha-(\mathbb{R}$\textbackslash$\{Y'+\mathbb{Q}(\bigcup\limits_{\beta<\alpha}K_\beta)\}))\neq\phi$ and therefore there exist $x_\alpha\in \mathbb{R}$\textbackslash$\{X'+\mathbb{Q}(\bigcup\limits_{\beta<\alpha}K_\beta)\}$ and $y_\alpha\in \mathbb{R}$\textbackslash$\{Y'+\mathbb{Q}(\bigcup\limits_{\beta<\alpha}K_\beta)\}$ such that $z_\alpha=x_\alpha+y_\alpha$.\\
	We set $X_\alpha=\mathbb{Q}x_\alpha+X'$ and $Y_\alpha=\mathbb{Q}y_\alpha+Y'$.\\
	Proceeding in this manner, we get two $\Omega$-sequences \{$X_\alpha\}_{\alpha<\Omega}$, \{$Y_\alpha\}_{\alpha<\Omega}$ of sets in $\mathbb{R}$ with the following properties :
	\begin{enumerate}[(i)]
		\item for each $\alpha<\Omega$, card($X_\alpha$) = card($Y_\alpha$) = card($\alpha$)+$\omega$.
		\item if $\beta<\alpha$, then $K_\beta\cap X_\beta=K_\beta\cap X_\alpha$ and  $K_\beta\cap Y_\beta=K_\beta\cap Y_\alpha$.
	\end{enumerate}
	Setting $E=\bigcup\limits_{\alpha<\Omega}X_\alpha$ and $F=\bigcup\limits_{\alpha<\Omega}Y_\alpha$, we find that $E$ and $F$ are rare sets which are also vector spaces over $\mathbb{Q}$ such that $\mathbb{R}=E+F$.\\
	We know that in any vector space over a field, if $U$ and $V$ are subspaces, then there exists a subspace $V'\subseteq V$ such that $U+V=U+V'$ and $U\cap V'=\{0\}$. Since $U\cap V$ is a subspace, $V'$ may be taken as a subspace of $V$ complemented to $U\cap V$. Thus finally in the above procedure, we may assume that $E\cap F=\{0\}$. This proves the theorem.
\end{proof}
It may be noted here that any perfect base is a point-meager Baire base (Th 3, I, Ch 5, [9]), so if it is linearly invariant, then Theorem 2.5 can be reproduced in an exactly similar manner in any Polish topological vector space.
\begin{theorem}
	Let $X$ be Polish topological vector space. Then under continuum hypothesis, in any perfect base ($X,\mathcal{C}$) satisfying c.c.c where $\mathcal{C}$ is a linearly invariant family having cardinality $\leq 2^{\aleph_0}$, there exist sets $R_1$ and $R_2$ such that 
	\begin{enumerate}[(i)]
		\item both $R_1$ and $R_2$ are rare sets.
		\item both $R_1$ and $R_2$ are vector spaces over the field $\mathbb{Q}$ of rationals.
		\item $X=R_1+R_2$ and $R_1\cap R_2=\{0\}$.
	\end{enumerate}
\end{theorem}
\begin{definition}[\cite{QM1979}]
	Two category bases ($X,\mathcal{C}$) and ($X,\mathcal{D}$) are said to be equivalent if $\mathcal{M}(\mathcal{C})=\mathcal{M}(\mathcal{D})$ and $\mathcal{B}(\mathcal{C})=\mathcal{B}(\mathcal{D})$.
\end{definition}
\begin{definition}[\cite{QM1979}]
	Two category bases ($X,\mathcal{C}$) and ($X,\mathcal{D}$) are called complementary if $X$ is representable as the union of two disjoint sets $M$ and $N$, where $M$ is $\mathcal{C}$-meager and $N$ is $\mathcal{D}$-meager.
\end{definition}
In our next result, we use the following theorem by Sander
\begin{theorem}[\cite{ASDG3425}]
	Two perfect translation bases ($X,\mathcal{C}$) and ($X,\mathcal{D}$) in any topological group which is a complete separable space and without isolated points are non-equivalent iff they are complementary.
\end{theorem}
The above theorem is also valid for any topological group $X$ with the property that each topologically dense subset of $X$ contains a countable dense subset [12].
\begin{theorem}
	Let $X$ be a Polish topological vector space. Then under continuum hypothesis, in any perfect translation base ($X,\mathcal{C}$) satisfying c.c.c where $\mathcal{C}$ is linearly invariant having cardinality $\leq 2^{\aleph_0}$, there exist two sets such that 
	\begin{enumerate}[(i)]
		\item both are rare sets.
		\item their sum is a Bernstein set in $X$.
		\item they are meager in any perfect translation base which is non-equivalent to ($X,\mathcal{C}$).
	\end{enumerate}
\end{theorem}
\begin{proof}
	According to Theorem 2.6, there exist rare sets $R_1$ and $R_2$ which are vector spaces over $\mathbb{Q}$ and $X=R_1+R_2$, $R_1\cap R_2=\{0\}$. Let \{$P_\alpha\}_{\alpha<\Omega}$ be an enumeration of all nonempty perfect sets in $X$ ($X$ being Polish, we know that the cardinality of all nonempty perfect sets in $X$ is $c$). It may be assumed that each of the subfamilies \{$P_\alpha:\alpha<\Omega$, $\alpha$ is an even ordinal\} and \{$P_\alpha:\alpha<\Omega$, $\alpha$ is an odd ordinal\} also contains all nonempty perfect subsets of $X$.\\
	Using transfinite recurssion, we now construct an injective $\Omega$-sequence of points in $R_2$ in the following manner :\\
	Suppose for any ordinal $\alpha<\Omega$, the partial sequence \{$y_\beta:\beta<\alpha$\} is already constructed. As $\bigcup\limits_{\beta<\alpha}(R_1+y_\beta)$ is a rare set because of continuum hypothesis and translation invariance  of $\mathcal{C}$ and every rare set in a perfect base satisfying c.c.c is Marczewski singular (Th 8, II, Ch 5, [9]), so $P_\alpha-\bigcup\limits_{\beta<\alpha}(R_1+y_\beta)\neq\phi$. We choose a point $z$ from the set $P_\alpha-\bigcup\limits_{\beta<\alpha}(R_1+y_\beta)$. By virtue of the equality $X=R_1+R_2$, there exists $y\in R_2$ such that $z\in R_1+y$. We put $y_\alpha=y$.\\
	We define $T_2=\{y_\alpha:\alpha<\Omega$, $\alpha$ is an even ordinal\} and $T_3=\{y_\alpha: \alpha<\Omega$, $\alpha$ is an odd ordinal\}.\\
	Then both $T_2$ and $T_3$ have cardinality $2^{\aleph_0}$ and according to the construction, they are also rare sets. Moreover, for any even ordinal $\alpha<\Omega$,
	\begin{equation*}
		P_\alpha\cap (R_1+T_2)\supseteq P_\alpha\cap(R_1+y_\alpha)\neq\phi
	\end{equation*}
and likewise for any odd ordinal $\alpha<\Omega$,
\begin{equation*}
		P_\alpha\cap (R_1+T_3)\supseteq P_\alpha\cap(R_1+y_\alpha)\neq\phi.
\end{equation*}
As $R_1\cap R_2=\{0\}$, so $(R_1+T_2)\cap(R_1+T_3)=\phi$ and both $R_1+T_2$ and $R_1+T_3$ are Bernstein sets in $X$.\\
Let ($X,\mathcal{D}$) be any perfect translation base such that ($X,\mathcal{C}$) and ($X,\mathcal{D}$) are nonequivalent. Then they are complementary (by Th 2.9) and accordingly $X$ can be expressed as the disjoint union of sets $M$ and $N$ where $M$ is $\mathcal{C}$-meager and $N$ is $\mathcal{D}$-meager. Obviously then the two rare sets constructed above are both $\mathcal{D}$-meager.\\
Hence the theorem. 
\end{proof}

Thus we find from the above theorem that under continuum hypothesis in any Polish topological vector space $X$, there exists for any perfect base ($X,\mathcal{C}$) where $\mathcal{C}$ is linearly invariant having cardinality atmost $2^{\aleph_0}$ and satisfying c.c.c, two  rare sets whose Minkowski's sum is non-Baire in any perfect base. In addition, if ($X,\mathcal{C}$) is also a translation base, then each of these rare sets is also meager in any perfect translation base which is nonequivalent (or, complementary) to ($X,\mathcal{C}$).
%\textbf{Acknowledgments: } The authors would like to express their gratitude to the learned referees who have given numerous valuable suggestions towards the improvement of the initial version of this article. 
\bibliographystyle{plain}

\begin{thebibliography}{99}
\Large{
	\bibitem{DFTR4577}
P. Erd\H{o}s, K. Kunen \& R. D. Mauldin, \emph{Some additive properties of sets of real numbers}, Fund. Math., Vol. CXIII, no. 3 (1981), pp 187-199.
\medskip	
\bibitem{SABM1979}
E. Grzegorek, \emph{Solution of a problem of Banach on $\sigma$-fields without continuous measures}, Bull. Acad. Polon. Sci. Ser. Sci. Math. Vol 28, 1980, pp 7-10.
\medskip
\bibitem{SAWE1254}
A. B. Kharazishvili, \emph{Non measurable sets and functions}, Elsevier, 2004.
\medskip
	\bibitem{AERT8765}
.................., \emph{Applications of point set theory in Real Analysis}, Kluwer Academic Publishers, Dordrecht, 1998. 
 \medskip
 \bibitem{QM1976}
..............., \emph{Topics in Measure Theory and Real Analysis}, Atlantis Press, World Scientific, 2009.
\medskip
\bibitem{ADFR2456}
.............., \emph{On measurability of algebraic sums of small sets}, Studia Scientiarum Mathematicarum Hungarica, Vol 45, no 2, 2008, pp 433-442.
 \medskip
 
 \bibitem{SWER7896}
 K. Kuratowski, \emph{Topology}, Vol 1, Academic Press, New York-London, 1966.
 \medskip
 \bibitem{KNJ234}
 A. W. Miller, \textit{Special subsets of the real line}, in: Handbook of Set-Theoretic Topology,  North-Holland Publ. Co., Amsterdam, 1984.
\medskip
 
 \bibitem{QM1979}
J. C. MorganII, \emph{Point set theory}, Marcel Dekker,Inc,1990.
 \medskip
 \bibitem{QM1976}
 ........, \emph{Baire category from an abstract viewpoint}, Fund. Math. 94(1976), $13$-$23$.
  \medskip
 \bibitem{AGT1974}
 ........, \emph{Infinite games and singular sets}, Collaq. Math. 29, 1974, 7-17.
 \medskip
 \bibitem{ASDG3425}
 W, Sander, \emph{A decomposition theorem}, Proc. Amer. Math. Soc., Vol 83, No 2, Nov 1981, 553-554
 \medskip
 \bibitem{DFR345}
 K. Schilling, \emph{Some category bases which are equivalent to topologies}, Real Analysis Exchange, Vol. 14, No. 1 (1988-1989), 210-214.
}
 
\end{thebibliography}

	\end{document}